\documentclass{article}
\usepackage[latin1]{inputenc}
\usepackage{a4,amsmath,amssymb,theorem}
\usepackage{harvard}
\newtheorem{theorem}{Theorem}
\newtheorem{corol}[theorem]{Corollary}
\newtheorem{rem}[theorem]{Remark}
{\theorembodyfont{\rm} \newtheorem{defin}[theorem]{Definition}}
\newtheorem{lemma}[theorem]{Lemma}
\newtheorem{prop}[theorem]{Proposition}

\newcommand{\D}{\displaystyle}

\newcommand{\ipl}{\langle}
\newcommand{\ipr}{\rangle}

\newcommand{\eop}{\mbox{} \hfill $\blacksquare$}

\def\N{\mathcal{N}}

\def\R{\mathcal{R}}

\newcommand{\cat}[3]{#1 \stackrel{#2}{\mathbf{\sqcap}}#3}
\newcommand{\eps}{\varepsilon}

\newcommand{\lbd}{\lambda}

\begin{document}

\title{On the value function for nonautonomous optimal control problems with
infinite horizon}
\author{J.\,Baumeister \\
\normalsize Department of Mathematics, Goethe University, \\
\normalsize Robert-Mayer Str. 6-10, D-60054 Frankfurt a.M., Germany \\[1ex]
A.\,Leit\~ao\footnote{Partially supported by CNPq, grants
305823/2003-5 and 478099/2004-5} \\
\normalsize Department of Mathematics, \\
\normalsize Federal University of St. Catarina,
88040-900 Florianopolis, Brazil \\[1ex]
G.N.\,Silva\footnote{Partially supported by CNPq and FAPESP.} \\
\normalsize Department of Computer Sciences and Statistics, \\
\normalsize Universidade Estadual Paulista, 15054-000 S.J. Rio Preto, Brazil}
\date{}
\maketitle

\begin{abstract}
In this paper we consider nonautonomous optimal control problems
of infinite horizon type, whose control actions are given by
$L^1$-functions. We verify that the value function is locally
Lipschitz. The equivalence between dynamic programming
inequalities and Hamilton-Jacobi-Bellman (HJB) inequalities for
proximal sub (super) gradients is proven.
Using this result we show that the value function is a Dini
solution of the HJB equation. We obtain a verification result
for the class of Dini sub-solutions of the HJB equation and
also prove a minimax property of the value function with
respect to the sets of Dini semi-solutions of the HJB equation.
We introduce the concept of viscosity solutions of the
HJB equation in infinite horizon and prove the equivalence between
this and the concept of Dini solutions. In the appendix we provide
an existence theorem.
\end{abstract}
\medskip

\setcounter{footnote}{1}
%
%
%
%------------------------------------------------------------------------------
\section{Introduction} \label{sec:intro}

\subsubsection*{Formulation of the optimal control problem}

Given the initial values $(\tau,z) \in [0,\infty) \times \R^n$,
let us consider the following optimal control problem:
$$  \left\{ \begin{array}{l}
   \mbox{Minimize }\ J(u; \tau,z) := \D\int_{\tau}^{\infty} e^{-\delta t}
                     l(t,x(t),u(t)) dt \\
   \mbox{subject to } \\[1ex]
   u \in U_{ad}[\tau,\infty):=\big\{ w \in L^1_{loc}([\tau,\infty); \R^m)\
   |\ w(t) \in \Omega \mbox{ a.e. in } [\tau,\infty) \big\} \\
   x(t) = z + \D\int_{\tau}^t f(s,x(s),u(s)) ds ,\ t \in [\tau,\infty) .
   \end{array} \right. $$
The problem above is called $P_{\infty}(\tau,z)$. Here $l: [0,\infty)
\times \R^n \times \R^m \to \R$, $f: [0,\infty) \times \R^n \times
\R^m \to \R^n$ are given functions; $\Omega \subset \R^m$ is the set
of admissible control actions; $\delta > 0$ denotes the discount
rate. An admissible control strategy $u$ is a $L^1_{loc}$-function
satisfying the control constraint $u(t) \in \Omega$ a.e. in
$[\tau,\infty)$. By $L^1_{loc}$ we mean the space of locally
integrable functions $w:[0,\infty) \to \R^n$. We call state trajectory
an absolutely continuous functions $x:[0,\infty)\rightarrow \R^n$ that
is a solution of the integral equation in $P_{\infty}(\tau,z)$.

A pair of functions $(x(\cdot),u(\cdot))$, consisting of an admissible
control $u$ and the corresponding state trajectory $x$, is called an
admissible process. We are specially interested in the set of admissible
processes with finite costs:
\begin{multline*}
{\cal D}_{{\tau},z;\infty} := \big\{ (x,u)\ |\ u \in
U_{ad}[\tau,\infty); ~~ x(t) = z + \int_{\tau}^t f(r,x(r),u(r)) dr; \\
J(u;\tau,z) < \infty \big\} .
\end{multline*}

We stress the fact that the cost functional in problem
$P_{\infty}(\tau,z)$ is of infinite horizon type and that both the
cost functional and the dynamic are allowed to be time dependent.
This sort of control problem is particularly interesting for
applications related to biological and economic sciences, since no
reasonable bound can be placed on the time horizon. Several
applications of this nature can be found in
\cite{CaHa,Cl,SS,ST,Fl,BaLe}, among others.

\subsubsection*{Discussion of the main contributions}

We are mainly concerned with the investigation of the value function for
the family of control problems $P_{\infty}(\tau,z)$.

We managed to extend to nonautonomous optimal control problems of
infinite horizon type some classical results of the dynamic
programming theory, like dynamic programming inequalities,
verification results, characterization of the
value function as both a Dini and viscosity solution of the HJB equation.%
\footnote{For an account of different types of solutions of the
HJB equation, see \cite{BaCa,CLSW,CrLi}.}

The difficulty in the formulation and analysis of the HJB equation for
problem $P_{\infty}(\tau,z)$ arises from the choice of the boundary
condition at infinity for the partial differential equation. It turns out
that this boundary condition can be formulated as uniform convergence
to zero, as $t \to \infty$, in each compact set of $\R^n$. Under the
assumption of existence of invariant sets for the trajectories, this
decay condition is exactly what one needs to prove the verification
result and the minimax property of the value function among Dini
semisolutions of the HJB equation.

Another feature of the present approach is the choice of controllers
as $L^1$-functions rather than $L^\infty$-functions, which are more common
in the existing literature (see, e.g., \cite{BaCa,dLi}).

The main tool used in this text to analyze the HJB equation in
infinite horizon is the concept of Dini solutions. The analysis of
Dini solutions of the HJB equation for finite horizon is
considered in \cite{CLSW,RaVi,VW}, among others. In finite
horizon, the analysis of the HJB equation for nonautonomuos
control problems is developed in \cite{RaVi,FPR}.

Dynamic programming for autonomous infinite horizon control
problems is considered by many authors \cite{Ha,BaCa,dLi}.
\cite{Ha} considers some economic applications of infinite horizon
which are modeled by control problems with continuous controllers.
 \cite{BaCa} characterizes the value
function as a viscosity solution of the HJB equation. They also
provide verification results and minimax properties. \cite{dLi}
obtains some regularity results for the value function and also
characterizes it as a stable viscosity solution of the HJB
equation.

        However, dynamic programming for nonautonomuos control problems of
infinite horizon has not been exploited in the literature. This
paper makes a first effort trying to close up this gap by setting up a
framework in which the dynamic programming approach can still be carried
out.

        As part of our developments we address the question of
existence of optimal processes for $P_{\infty}(\tau,z)$. An
existence theorem for problems of infinite horizon type is given
by \cite{Ba}. However, this result is not applicable in our
situation.

\subsubsection*{Outline of the paper}

        The paper is organized as follows. In Section~\ref{sec:lipschitz},
we verify some properties of the value function, including (local)
Lipschitz continuity. In Section~\ref{sec:dini-hjb} we introduce
several concepts of generalized derivatives and gradients. The
main purpose of this section is the definition the Dini solutions
of the HJB equation.
In Section~\ref{sec:monotonicity} we analyze some monotonicity properties of
functions related to the HJB equation.
In Section~\ref{sec:hjb-inequal} we prove the equivalence between certain
HJB inequalities for Dini-gradients and proximal-gradients. In the
sequel we use this result to prove that the value function is a Dini
solution of the HJB equation.
In section~\ref{sec:verif-mm} we derive a verification result for the
Dini sub-solutions of the HJB equation and prove a minimax property for
the value function (the uniqueness of Dini solutions follows from this
property).
In Section~\ref{sec:visc-sol} we verify the equivalence between the
concepts of Dini and viscosity solutions of the HJB equation. This
characterizes the value function as a viscosity solution of the
HJB equation.
The Appendix is devoted to discussing the issue of existence of optimal
processes for $P_{\infty}(\tau,z)$.
%
%
%
%------------------------------------------------------------------------------
\section{Lipschitz continuity of the value function} \label{sec:lipschitz}
\setcounter{equation}{0}

Let us consider again the family of optimal control problems
$P_\infty(\tau,z)$, for initial conditions $(\tau,z) \in \R \times
\R^n$. In the sequel we shall consider problem
$P_{\infty}(\tau,z)$ under the following assumptions:
\begin{itemize}
\item[\it A1)] There exists $K_1 > 0$, such that for every fixed $t
\in [0, \infty)$, $u \in \Omega$
$$ |l(t,x,u) - l(t,z,u)| + |f(t,x,u) - f(t,z,u)| \le K_1 |x - z| , $$
holds for all $x, z \in \R^n$;
\item[\it A2)] There exists a positive scalar $K_2$ such that $f$ and $l$
satisfy the linear growth condition
$$ |l(t,x,u)| + |f(t,x,u)| \le K_2 (1 + |x|) , $$
for all $t \in [0,\infty)$, $x \in \R^n$, $u \in \Omega$;
\item[\it A3)] The functions $l(\cdot, \cdot, \cdot)$,
$f(\cdot, \cdot, \cdot)$ are continuous;
\item[\it A4)]  The set of points $\bar{f} (t,x,\Omega) :=
\{ (e^{-\delta t} l(t,x,v), f(t,x,v))^T\ |\ v \in \Omega \}$ is a
convex set in $\R^{n+1}$ for all $t$, $x$ (here $T$ means transposition).
\item[\it A5)] $\Omega$ is compact;
\item[\it A6)] Given $(\tau,z) \in [0, \infty) \times \R^n$, there
exists a bounded (invariant) set ${\cal S}_{\tau,z} \subset \R^n$ such
that every admissible process $(x(\cdot),u(\cdot)) \in {\cal D}_{\tau,z;
\infty}$ satisfies $x(t) \in {\cal S}_{\tau,z}$, for $t \ge \tau$.
\end{itemize}

\noindent
The function $\bar{f}$ in {\it A4)} is called the extended velocity vector
related to the control problem $P_{\infty}(\tau,z)$.

We start by defining a key function in the theory of dynamic programming,
namely, the value function.

\begin{defin} \label{def:vf}
The application
\begin{eqnarray}
V: [0,\infty) \times \R^n & \longrightarrow & \R \nonumber\\
(\tau,z) & \longmapsto &
                 \inf\{J(u;\tau,z)\, |\, (x,u) \in {\cal D}_{\tau,z;\infty} \}
\end{eqnarray}
is called the {\em value function} associated with the family of problems
$P_\infty(\tau,z)$.
\end{defin}

Next we define the Hamilton-function as the application
\begin{eqnarray}
 H:[0,\infty)\times\R^n\times\R^n\times\R^m & \longrightarrow & \R \nonumber\\
   (t,x,\lbd,u) & \longmapsto & \ipl \lbd ,\, f(t,x,u) \ipr \ + \
   e^{-\delta t} l(t,x,u) .
\end{eqnarray}
We also define the (control independent) auxiliary function
\begin{eqnarray} \label{eq:def-hamilt}
 {\cal H} : [0,\infty) \times \R^n \times \R^n & \longrightarrow & \R
 \nonumber\\
 (t,x,\lbd) &  \longmapsto & \inf_{u \in \Omega} \{ H(t,x,\lbd,u) \}
\end{eqnarray}
Instead of {\it A6)}, we shall consider in this section the assumption
\begin{itemize}
\item[\it \~A6)] The Lipschitz constant $K_1$ in {\it A1)} and the discount
rate $\delta$ sa\-tis\-fy $K_1 - \delta < 0$. Furthermore, given
$(\tau,z) \in [0, \infty) \times \R^n$, there exists a neighborhood
${\cal V}_{\tau,z}$ of $(\tau,z)$ and a bounded (invariant) set
${\cal S}_{\tau,z} \subset \R^n$ such that for every initial condition
$(\tilde{\tau},\tilde{z}) \in {\cal V}_{\tau,z}$, the admissible processes
$(x(\cdot),u(\cdot)) \in {\cal D}_{\tilde{\tau},\tilde{z};\infty}$
satisfy $x(\cdot) \in {\cal S}_{\tau,z}$.
\end{itemize}

\begin{rem}
The assumption concerning the relationship between the Lipschitz
constant $K_1$ and the discount rate $\delta$ looks somehow
artificial. Nevertheless, it has been used by several authors,
see, e.g., \cite{BaCa,Wi,Co}.
\end{rem}

Under the  extra assumption above, it is possible to verify some
important properties of the value function $V$, which we summarize
in Lemma~\ref{lem:vf-regular}. First we state an auxiliary lemma.
The proof of this result does not contain new ideas and will be
omitted.

\begin{lemma} \label{lem:xyest}
Suppose the assumptions {\it A1), A2)} are satisfied. Let $(\tau,z)$,
$(\tau,z_1) \in [0,\infty) \times B_R(0)$ and $u \in U_{ad}[\tau,\infty)$.
Let $x$, $x_1$ be the associated states. Then
\begin{equation} \label{eq:dis}
|x(t) - x_1(t)| \le e^{K_1 (t-\tau)} |z - z_1| ,\ t \ge \tau .
\end{equation}
Furthermore, we have
\begin{equation} \label{eq:growth}
|x(t) - z| \le K_2 (1 + R) e^{K_2 (t-\tau)} |t-\tau| ,\ t \ge \tau .
\end{equation}
\end{lemma}

In the next lemma we need to combine (admissible) controls defined in
different time intervals. Let $u_1 \in U_{ad}[s_1,\infty]$, $u_2 \in
U_{ad} [s_2,\infty]$, where $0 \le s_1 \le s_2 < \infty$. We define
the {\em concatenation}
$$  \cat{u_1}{s_2}{u_2} \ := \ \left\{
    \begin{array}{r@{\ }l}
      u_1(t) &, t \in [s_1,s_2) \\
      u_2(t) &, t \in [s_2,\infty)
    \end{array} \right. . $$

\begin{lemma} \label{lem:vf-regular}
Let assumptions A1),\dots, A5) and \~A6) hold. Then we have
\begin{itemize}
\item [a) ] The value function is well defined on $[0,\infty)\times\R^n$;
\item [b) ] The value function is localy Lipschitz continuous in its
domain of definition.
\end{itemize}
\end{lemma}
\begin{proof} Notice that the existence of an optimal process is guaranteed
by Theorem~\ref{satz:exist}. Than, {\it a)} follows.
The assertion {\it b)} is the most interesting one. To prove it, it is
enough to verify that for every $T>0$, $R>0$, there exists a constant
$C$ such that for all $(s_1,z_1)$, $(s_2,z_2) \in [0,T] \times
\overline{B}_R(0)$ the following inequality holds:
\begin{equation} \label{eq:vf-lip}
|V(s_1,z_1) -V(s_2,z_2)| \le C (|s_1-s_2| +|z_1-z_2|) .
\end{equation}
Since $[0,T] \times \overline{B}_R(0)$ is compact, it follows from
{\it \~A6)} that there exists a bounded invariant set ${\cal S}_{T,R}
\subset \R^n$ such that for every initial condition $(\tau,z) \in
[0,T] \times \overline{B}_R(0)$, the admissible processes
$(x(\cdot),u(\cdot)) \in {\cal D}_{\tau,z;\infty}$ satisfy
$x(\cdot) \in {\cal S}_{T,R}$, i.e.,
\begin{equation} \label{eq:traj-estim}
|x(t)| \le M_{T,R} ,\ t \in [\tau,\infty) .
\end{equation}
Let $(s_1,z_1)$, $(s_2,z_2) \in [0,\infty) \times B_R(0)$ be given,
with $s_1 \le s_2$. Given $\eps > 0$, choose $u_1 \in U_{ad}[s_1,\infty)$,
$u_2 \in U_{ad}[s_2,\infty)$ such that
$$  V(s_i,z_i) \; \ge \; J(u_i; s_i,z_i) \, - \, \eps\, ,\ i = 1,2\, . $$
Define $\widetilde{u}_1 := \cat{u_1}{s_2}{u_2}$, $\widetilde{u}_2 :=
{u_1}_{|[s_2,\infty)}$. Let $x_1$, $x_2$, $\widetilde{x}_1$, $\widetilde{x}_2$
be the corresponding states with $x_1(s_1) = \widetilde{x}(s_1) = z_1$,
$x_2(s_2) = \widetilde{x}_2(s_2) = z_2$.%
\footnote{Notice that $J(\widetilde{u}_1; s_1, z_1)$ and/or
$J(\widetilde{u}_2; s_2, z_2)$ may be infinite.}
Then, it follows
\begin{eqnarray*}
V(s_1,z_1) - V(s_2,z_2) & \le & V(s_1,z_1) - J(u_2; s_2,z_2) + \eps
                                \nonumber \\
& \le & J(\widetilde{u}_1; s_1,z_1) - J( u_2; s_2,z_2) + \eps
\end{eqnarray*}
and by symmetry
\begin{eqnarray}
| V(s_1,z_1) - V(s_2,z_2) | & \le & \eps +
   | J(\widetilde{u}_1;s_1,z_1) - J(u_2;s_2,z_2) | \nonumber \\
& & + | J(u_1;s_1,z_1) - J(\widetilde{u}_2;s_2,y_2) | . \label{eq:vf-absch}
\end{eqnarray}
To obtain (\ref{eq:vf-lip}) we have to estimate the terms
$$|J(\widetilde{u}_1; s_1,z_1) - J(u_2; s_2,z_2)|\, ,\ \
  |J(u_1; s_1,z_1) - J(\widetilde{u}_2; s_2,y_2)| . $$
From the definition of $J$ and assumptions {\it A2)}, {\it A3)} it follows that
\begin{eqnarray}
|J(\widetilde{u}_1; s_1,z_1) - J(u_2; s_2,z_2)| & \le&
  \Big| \int_{s_1}^{s_2} e^{-\delta r}
  l(r,\widetilde{x}_1(r),\widetilde{u}_1(r))dr \Big| \nonumber \\
&\hspace{-.2cm} & \hspace{-3.5cm}
  + \, \Big| \int_{s_2}^\infty e^{-\delta r}
  l(r,\widetilde{x}_1(r),\widetilde{u}_1(r)) -
  e^{-\delta t} l(r,x_2(r),u_2(r)) dr \Big| \nonumber \\
&\hspace{-7.2cm} \le & \hspace{-3.5cm}
  K_2 \int\limits_{s_1}^{s_2}\! e^{-\delta r} (1+|x_1(r)|) dr +
  K_1 \int\limits_{s_2}^\infty e^{-\delta r}
  |\widetilde{x}_1(r)-x_2(r)| dr . \label{eq:vf-lip-abs}
\end{eqnarray}
From the Gronwall lemma, the first term on the right hand side of
(\ref{eq:vf-lip-abs}) can be estimated by $\widetilde{C} |s_1 - s_2|$,
with some constant $\widetilde{C}$ depending on $K_2$, $T$ and $R$.
We now estimate the last term. It follows from (\ref{eq:dis}) that
\begin{eqnarray*}
|\widetilde{x}_1(r)-x_2(r)| & \le &
        e^{K_1 (r-s_2)} |\widetilde{x}_1(s_2)-z_2| \\
& \le & e^{K_1 (r-s_2)} \big( |\widetilde{x}_1(s_2)-z_1| + |z_1-z_2| \big),
\end{eqnarray*}
for $r \ge s_2$. Furthermore, it follows from (\ref{eq:growth}) that
$|\widetilde{x}_1(s_2)-z_1| \le C |s_2-s_1|$, with $C = K_2(1+R)e^{K_2T}$.
Hence,
$$  |\widetilde{x}_1(r) - x_2(r)| \le
    C e^{K_1 (r-s_2)} (|s_1 -s_2| + |z_1 -z_2|) ,\ r > s_2 . $$
Substituting in (\ref{eq:vf-lip-abs}) and using the assumption $K_1<\delta$,
we obtain
$$ |J(\widetilde{u}_1; s_1,z_1) - J(u_2; s_2,z_2)| \le
   C\, (|s_1 -s_2| + |z_1 -z_2|) . $$
The term $|J(u_1; s_1,z_1) - J(\widetilde{u}_2; s_2,y_2)|$ can be
estimated analogously and it follows from (\ref{eq:vf-absch}) that
$$ |V(s_1,z_1) - V(s_2,z_2)| \le  \eps + C\, (|s_1-s_2| + |z_1 - z_2|), $$
where the constant $C$ is independent of $z_1$, $z_2$, $s_1$, $s_2$, $\eps$
and grows exponentially with $T$. By taking the limit $\eps \downarrow 0$,
the lemma follows. \eop
\end{proof}
%
%
%
%------------------------------------------------------------------------------
\section{Dini solutions of the HJB equation} \label{sec:dini-hjb}
\setcounter{equation}{0}

In this section we discuss a first concept of solution for the HJB
equation

\begin{equation} \label{eq:hjbg}
\partial_t v(t,x) \, + \, {\cal H}(t,x,\partial_x v(t,x)) \ = \ 0 ,
\ (t,x) \in (0,\infty) \times \R^n ,
\end{equation}
where ${\cal H}$ is the function defined in (\ref{eq:def-hamilt}).

        As it is well known, one should not consider $C^1$-solutions
of (\ref{eq:hjbg}) since such solutions do not exist in general: no
matter how regular the control problem is, the value function might
not be differentiable everywhere. Thus, a concept of nonsmooth
solution for the HJB equation is necessary in this context and will
be introduced in the sequel.

\begin{rem}
The HJB equation (\ref{eq:hjbg}) is of nonstationary type since the
control problem is nonautonomous. If the functions $f$, $l$ do not
depend explicitly on $t$, then, by a simple time transformation,
(\ref{eq:hjbg}) can be transformed into a stationary equation.
\end{rem}

For the rest of this section let $\Phi$ be a
mapping from $[0,\infty) \times \R^n$ to $\R$. This mapping should be
considered as candidate for the value function.

\begin{defin} \label{def:dini-der}
The {\em lower Dini derivative} of a function $\Phi$ at $(t,x)$ in
the direction $(\alpha,\xi) \in \R^{1+n}$, is defined by
$$  D^- \Phi(t,x;\alpha,\xi):= \liminf\limits_{r\downarrow 0; w\to\xi}
\frac{\D\Phi(t + r\alpha, x + rw) - \Phi(t,x)}{\D r} . $$
\end{defin}

        Before proceeding we make a remark on the definition of
lower Dini directional derivatives. Recently, it has been widely
accepted that a lower Dini directional derivative must be defined
through the lower limit with changes in both variables $(t,x)$,
as it is defined above. However, since we are dealing with
Lipschitz functions, the more general definition reduces to a
limit on $r \downarrow 0$ alone, as according to \cite[Lemma~5.1]{VW}.

        One should notice that the change of the $\liminf$ in Definition~%
\ref{def:dini-der} by a $\limsup$ allows the definition of the
\textit{upper Dini derivative} of $\Phi$ (denoted by $D^+ \Phi$).

        Next we define the Dini generalized gradients of a function.

\begin{defin} \label{def:dini-grad}
The {\em Dini sub-gradient} of a function $\Phi$ at $(t,x) \in [0,\infty)
\times \R^n$ is defined by
$$ \partial_D \Phi(t,x) := \{ (\alpha,\xi) \in \R^{1+n} ;\
   D^- \Phi(t,x; u,v) \ge u \alpha + \ipl v , \xi \ipr,\ \forall
   (u,v) \in \R^{1+n} \} . $$
Analogously, the {\em Dini super-gradient} of $\Phi$ at $(t,x) \in
[0,\infty) \times \R^n$ is defined by
$$ \partial^D \Phi(t,x) := \{ (\alpha,\xi) \in \R^{1+n} ;\
   D^+ \Phi(t,x; u,v) \le u \alpha + \ipl v , \xi \ipr,\ \forall
   (u,v) \in \R^{1+n} \} . $$
\end{defin}

It is easy to see that $\partial^D \Phi(t,x) = -\partial_D (-\Phi)(t,x)$.
Now we are ready to introduce the concept of Dini semi-solutions of the
HJB equation, which is our main interest in this paper.

\begin{defin} \label{def:dini-sol}
We say that a function $\Phi$ is a {\em Dini sub-solution} of the
HJB equation if it is continuous and satisfies
\begin{enumerate}
\item[a)] for all $(\tau,z) \in [0,\infty) \times \R^n$,
\begin{equation} \label{eq:dini-deriv-grad-hjb}
\alpha + {\cal H}(\tau,z,\xi) \ge 0 ,\
\forall (\alpha,\xi) \in \partial^D \Phi(\tau,z) ,
\end{equation}
\item[b)] $\Phi(\tau,z) \to 0,$ as $\tau\to\infty$, uniformly on each
compact subset $K \subset \R^n$.
\end{enumerate}
If $\Phi$ is continuous, satisfies {\it b)}, and further
\begin{enumerate}
\item[c)] for all $(\tau,z) \in [0,\infty) \times \R^n$,
\begin{equation} \label{eq:dini+deriv-grad-hjb}
\alpha + {\cal H}(\tau,z,\xi) \le 0 ,\
\forall (\alpha,\xi) \in \partial_D \Phi(\tau,z) ,
\end{equation}
\end{enumerate}
then it is called {\em Dini super-solution} of the HJB equation. If
$\Phi$ satisfies {\it a)}, {\it b)}, {\it c)}, is called {\em Dini
solution} of the HJB equation.
\end{defin}

We close this section introducing another concept of generalized
gradients, namely the proximal gradients. This will allow us to
establish a relationship between the concepts of Dini and viscosity
solutions of the HJB equation.

\begin{defin}
The {\em proximal sub-gradient} of a function $\Phi$ at $(t,x)$, denoted
by $\partial_P \Phi(t,x)$, is the set of all vectors $(\alpha,\xi) \in
\R^{1+n}$, such that there exists $\sigma >0$ and a neighborhood $U$ of
$(t,x)$ with
\begin{equation} \label{def:prox-subgrad}
\Phi(\tau,y) \ge \Phi(t,x) + \alpha (\tau-t) + \ipl \xi,y-x \ipr -
\sigma (\|\tau-t\|^2 + \|y-x\|^2) ,
\end{equation}
for all $(\tau,y) \in U$. Analogously, the {\em proximal super-gradient}
of a function $\Phi$ at $(t,x)$, denoted by $\partial^P \Phi(t,x)$, is
the set of all vectors $(\alpha,\xi) \in \R^{1+n}$, such that there
exists $\sigma >0$ and a neighborhood $U$ of $(t,x)$ with
\begin{equation} \label{def:prox-supergrad}
\Phi(\tau,y) \le \Phi(t,x) + \alpha (\tau-t) + \ipl \xi,y-x \ipr +
\sigma (\|\tau-t\|^2 + \|y-x\|^2) ,
\end{equation}
for all $(\tau,y) \in U$.
\end{defin}

        It is worth noticing that, the proximal super-gradient can be
alternatively defined by $\partial^P \Phi(t,x) = -\partial_P (-\Phi)(t,x)$.
Notice also that $\partial_P \Phi(t,x) \subset \partial_D \Phi(t,x)$.
%
%
%
%----------------------------------------------------------------------------
\section{Monotonicity properties} \label{sec:monotonicity}
\setcounter{equation}{0}

In this section we verify some monotonicity properties related to the
solutions of HJB inequalities. We close the section with a monotonicity
property involving the value function.

Let $f, l$ be as defined in Section~\ref{sec:intro}. For this section
we suppose that assumptions {\it A1)}, \dots, {\it A5)} hold. Further,
let $\Phi:[0,\infty) \times \R^n \to \R$ be a continuous function.
The pairs $(\Phi,f)$ are called systems. A system $(\Phi,f)$ is said to
be {\em weakly decreasing} when, for every $z\in \R^n$, there exists an
admissible process $(x,u) \in {\cal D}_{0,z;\infty}$ satisfying
$\Phi(t,x(t)) \le \Phi(0,x(0))$, $t \ge 0$. We say that a system
$(\Phi,f)$ is {\em strongly increasing} in $[0,\infty)$ when for every
interval $[a,b] \subset [0,\infty)$, each admissible process $(x,u)
\in {\cal D}_{a,x(a),\infty}$ satisfy $\Phi(t,x(t)) \le \Phi(b,x(b))$,
$t \in [a,b]$. Weakly increasing and strongly decreasing systems are
defined in an analogous way.

It is easy to check that $(\Phi,f)$ is strongly increasing in $[0,\infty)$
iff the function $t \mapsto \Phi(t,x(t))$ is non-decreasing for every
admissible process $(x,u)$ of $P_{\infty}(0,x(0))$. In the sequel we
analyze a relation between the monotonicity of systems and HJB inequalities.

\begin{prop} \label{prop:monotonicity}
Let $\varphi:[0,\infty) \times \R^n \to \R$ be a continuous function and
${\cal H}$ be the Hamiltonian function in (\ref{eq:def-hamilt}). The
following assertions hold true:
\begin{description}
\item[(a)] A function $[s,\infty) \ni t \mapsto \varphi(t,x(t)) -
\int_{t}^{\infty} e^{-\delta \sigma} l \big(\sigma,x(\sigma),u(\sigma) \big)
d\sigma \in \R$ is non-decreasing along every process
$(x,u) \in {\cal D}_{s,z;\infty}$ iff
\begin{equation} \label{eq:HJB-P}
\inf\limits_{(\alpha,\xi) \in \partial_P \varphi(s,z)} \big\{
\alpha + {\cal H}(s,z,\xi)\,\big\} \ge 0,
\end{equation}
for all $(s,z) \in [0,\infty) \times \R^n$.
%%%%%%%
\item[(b)] A function $[s,\infty) \ni t \mapsto \varphi(t,x(t)) -
\int_{t}^{\infty} e^{-\delta \sigma} l \big(\sigma,x(\sigma),u(\sigma) \big)
d\sigma\in \R$ is non-increasing along every process
$(x,u) \in {\cal D}_{s,z;\infty}$ iff
\begin{equation} \label{eq:HJB-PM}
\sup\limits_{(\alpha,\xi) \in \partial_P \varphi(s,z)} \big\{
\alpha + {\hat{\cal H}}(s,z,\xi)\,\big\} \le 0,
\end{equation}
for all $(s,z) \in [0,\infty) \times \R^n$.
%%%%%%%
\item[(c)] A function $[s,\infty) \ni t \mapsto \varphi(t,x(t)) -
\int_{t}^{\infty} e^{-\delta \sigma} l \big(\sigma,x(\sigma),u(\sigma) \big)
d\sigma\in \R$ is non-increasing along some process
$(x,u) \in {\cal D}_{s,z;\infty}$ iff
\begin{equation} \label{eq:HJB-PMM}
\sup\limits_{(\alpha,\xi) \in \partial_P \varphi(s,z)} \big\{
\alpha + {\cal H}(s,z,\xi)\,\big\} \le 0,
\end{equation}
for all $(s,z) \in [0,\infty) \times \R^n$.
\end{description}
\end{prop}

\noindent
The notation ${\hat{\cal H}}$ means the substitution of the minimization
by a maximization in the definition of ${\cal H}$, i.e.,
$$ {\hat{\cal H}}(s,z,\xi) := \sup\limits_{u\in \Omega}
   \left\{\ipl \xi, f(s,z,u) \ipr + e^{-\delta s} l(s,z,u)\right\}. $$

\begin{proof}
We prove only assertion (a) of the proposition, the others being
analogous. Let $\phi: [0,\infty) \times \R^n \times \R \to \R$ be the
function defined by
$$ \phi(t,x,y) := \varphi(t,x) - y . $$
As a consequence of the proximal sub-gradient inequality, we have
\begin{equation} \label{eq:grad-p-ext}
( \alpha, \zeta, \zeta_0) \in \partial_P \phi(t,x,y) \
\Longleftrightarrow \
( \alpha, \zeta) \in \partial_P \varphi(t,x) \ \ {\rm and}\ \ \zeta_0 = -1 .
\end{equation}
Define the extended vector field $\hat{f}:=(-e^{-\delta\cdot} l, f)^T$.
Note that, because of {\it A1)}, \dots, {\it A5)}, the function $\hat{f}$
is continuous; $\hat{f}(t,x,\Omega)$ is a convex compact set for
$t,x \in [0,\infty) \times \R^n$; $\hat{f}$ satisfies a linear growth
condition with respect to $x$.

First we prove the necessity of condition (\ref{eq:HJB-P}). From the
hypothesis, follows that $\phi$ is monotone increasing along every
admissible process for $\hat{f}$. This is equivalent to the system
$(\phi,\hat{f})$ being strongly increasing. Then, it follows from
\cite[Exercise 4.6.4]{CLSW}, that
\begin{equation} \label{eq:grad-p-ext2}
\alpha + \inf\limits_{u\in \Omega} \{\ipl \zeta,f(t,x,u) \ipr -
\ipl \zeta_0, e^{-\delta t} l(t,x,u)\ipr \} \ge 0,
\end{equation}
for all $(t,x,y) \in [0,\infty) \times \R^n \times \R$, and all
$(\alpha,\zeta,\zeta_0) \in \partial_P \phi(t,x,y)$. This last inequality,
together with (\ref{eq:grad-p-ext}), guarantee that (\ref{eq:HJB-P}) is
satisfied for all $(t,x) \in [0,\infty) \times \R^n$ and all
$(\alpha,\zeta) \in \partial_P \varphi(t,x)$.

Next we prove the sufficiency of condition (\ref{eq:HJB-P}). Note that
this condition, together with (\ref{eq:grad-p-ext}), guarantee that
(\ref{eq:grad-p-ext2}) is satisfied for all $(t,x,y) \in [0,\infty)
\times \R^n \times \R$ and all $(\alpha,\zeta,\zeta_0) \in \partial_P
\phi(t,x,y)$. Now, it follows from \cite[Exercise 4.6.4]{CLSW}, that the
system $(\phi,\hat{f})$ is strongly increasing. This is equivalent to
the desired monotonicity of the function $[s,\infty) \ni t \mapsto
\varphi(t,x(t)) - \int_{t}^{\infty} e^{-\delta \sigma} l
\big(\sigma, x(\sigma), u(\sigma) \big) d\sigma \in \R$. \eop
\end{proof} \bigskip

In the sequel we verify that the real function defined by
$$ [s,\infty) \ni t \mapsto V(t,x(t)) - \int_{t}^{\infty} e^{-\delta
   \sigma} l \big(\sigma,x(\sigma),u(\sigma) \big) d\sigma\in \R $$
is non-decreasing along every admissible process $(x,u) \in {\cal
D}_{s,z;\infty}$. Choose a fixed $t \in [s,\infty)$ and $r > 0$.
Thus, for every process $(x,u) \in {\cal D}_{s,z;\infty}$, we have
\begin{eqnarray*}
& V(t+r,x(t+r)) - \int_{t+r}^{\infty} e^{-\delta \sigma}
  l(\sigma,x(\sigma),u(\sigma)) d\sigma \\
& - \left[ V(t,x(t)) - \int_{t}^{\infty} e^{-\delta \sigma}
  l(\sigma,x(\sigma),u(\sigma)) d\sigma \right] \\
& = V(t+r,x(t+r)) - V(t,x(t)) + \int_{t}^{t+r} e^{-\delta \sigma}
  l(\sigma,x(\sigma),u(\sigma)) d\sigma \ge 0;
\end{eqnarray*}
(the inequality follows from Bellman's optimality principle). The desired
mono\-tony follows now from this inequality. An immediate consequence of
this monotonicity (c.f. Proposition~\ref{prop:monotonicity}) is the fact
that the value function satisfies (\ref{eq:HJB-P}).
%
%
%
%----------------------------------------------------------------------------
\section{HJB inequalities} \label{sec:hjb-inequal}\setcounter{equation}{0}

In this paragraph we discuss some equivalences between HJB inequalities
for proximal sub-gradients and Dini sub-gradients. We conclude the section
by characterizing the value function as a Dini solution of the HJB equation.

\begin{prop} \label{prop:equiv}
Let assumptions {\it A1)}, \dots, {\it A5)} hold. If $\Phi:[0,\infty)
\times \R^n \to \R$ is a continuous function, then
\begin{description}
\item [(a)]
\hskip-0.5cm$\sup\limits_{(\alpha,\xi) \in \partial_P \Phi(s,z)}
\!\!\!\!\big\{ \alpha + {\cal H}(s,z,\xi) \big\} \le 0$, $\forall (s,z)$
$\Longleftrightarrow$
\hskip-0.5cm$\sup\limits_{(\alpha,\xi) \in \partial_D \Phi(s,z)}
\!\!\!\!\big\{ \alpha + {\cal H}(s,z,\xi) \big\} \le 0$, $\forall (s,z)$;
\item [(b)] 
\hskip-0.5cm$\inf\limits_{(\alpha,\xi) \in \partial_P \Phi(s,z)}
\!\!\!\!\big\{ \alpha + {\cal H}(s,z,\xi) \big\} \ge 0$, $\forall (s,z)$
$\Longleftrightarrow$
\hskip-0.5cm$\inf\limits_{(\alpha,\xi) \in \partial^D \Phi(s,z)}
\!\!\!\!\big\{ \alpha + {\cal H}(s,z,\xi) \big\} \ge 0$, $\forall (s,z)$.
\end{description}
\end{prop}

\begin{proof}
We start by proving (a). The sufficiency follows from the inclusion
$\partial_P \phi(t,x) \subset \partial_D \phi(t,x)$ for all
$(t,x)$ in the domain of $\phi$. In order to prove the necessity,
take an arbitrary $(\alpha,\beta) \in \partial_D\phi(s,z)$. It
follows from \cite[Proposition~3.4.5]{CLSW} the existence of
$(\alpha_n,\beta_n) \in (\alpha,\beta) + B_{\frac{1}{n}}(0)$ and
$(s_n,z_n) \in [0,\infty)\times\R^n$, for $n \in \N$, such that
$|s_n-s| + |z_n-z| < 1/n$, $(\alpha_n,\beta_n) \in \partial_P
\phi(s_n,z_n)$ and $|\phi(s_n,z_n) - \phi(s,z)| < 1/n$. Therefore,
$$ \alpha_n + {\cal H}(s_n,z_n,\beta_n) \le 0, \quad \forall n . $$
Since ${\cal H}$ is a continuous function, the desired result follows
by taking the limit $n \to \infty$.

Next we prove (b). The first inequality in (b) is satisfied for all
$(\alpha,\xi) \in \partial_P \varphi(s,z)$ if and only if the function
$$ [s,\infty) \ni t \mapsto \varphi(t,x(t)) - \int_{t}^{\infty}
   e^{-\delta \sigma} f_0 \big(\sigma,x(\sigma),u(\sigma) \big)
   d\sigma\in \R $$
is non-decreasing along every process $(x,u) \in {\cal D}_{s,z;\infty}$
(cf. Proposition~\ref{prop:monotonicity}). This is equivalent to the fact
that
$$ [s,\infty) \ni t \mapsto -\varphi(t,x(t)) +  \int_{t}^{\infty}
   e^{-\delta \sigma} f_0 \big(\sigma,x(\sigma),u(\sigma) \big)
   d\sigma\in \R $$
is non-increasing along every process $(x,u) \in {\cal D}_{s,z;\infty}$
This, however, is equivalent to
\begin{equation} \label{eq:gns1}
 \max\limits_{(\alpha,\xi) \in \partial_P (-\varphi)(s,z)} \big\{ \alpha +
{\tilde{\cal H}}(s,z,\xi)\,\big\} \le 0,
\end{equation}
for all $(s,z) \in [0,\infty) \times \R^n$, where
$$ \tilde{\cal H}(s,z,\xi) := \max\limits_{u\in \Omega}
   \left\{ \ipl \xi, f(s,z,u) \ipr - e^{-\delta s} f_0(s,z,u) \right\} . $$
Analogous as in the proof of item (a), we use \cite[Proposition~3.4.5]{CLSW}
to conclude that (\ref{eq:gns1}) is equivalent to
\begin{equation} \label{eq:gns2}
 \max\limits_{(\alpha,\xi) \in \partial_D (-\varphi)(s,z)} \big\{ \alpha +
{\tilde{\cal H}}(s,z,\xi)\,\big\} \le 0 ,
\end{equation}
for $(s,z) \in [0,\infty) \times \R^n$. Since $\partial_D(-\varphi)(s,z)
= -\partial^D \varphi(s,z)$, it is easy to see that (\ref{eq:gns2}) is
equivalent to
$$ \min\limits_{(\alpha,\xi) \in \partial^D \varphi(s,z)}
   \big\{ \alpha + {\cal H}(s,z,\xi)\,\big\} \ge 0, $$
proving (b). \eop
\end{proof} \bigskip

We are now ready to state and prove the main result of this section.

\begin{theorem} \label{th:vf-dini-sol}
Suppose the assumptions A1), \dots, A5) and \~A6) are satisfied. Then
the value function is a Dini solution of the HJB equation.
\end{theorem}

\begin{proof}
First we prove that the value function satisfies the decay condition in
Definition~\ref{def:dini-sol}. Since $V$ is locally Lipschitz (see
Lemma~\ref{lem:vf-regular}), it is enough to prove that
$\lim_{\tau\to\infty} V(\tau,z) = 0$, for each $z \in \R^n$. Let $z \in
\R^n$ and $\eps > 0$ be given. For each $\tau \ge 0$, choose a process
$(x_{\eps,\tau}(\cdot),u_{\eps,\tau}(\cdot)) \in {\cal
D}_{\tau,z;\infty}$ satisfying
$$ J(u_{\eps,\tau};\tau,z) - \eps \le V(\tau,z) \le
   J(u_{\eps,\tau};\tau,z) . $$
We now prove that
\begin{equation} \label{eq:lim-J}
\lim_{\tau\to\infty} J(u_{\eps,\tau};\tau,z) = 0 .
\end{equation}
Indeed, from {\it A2)} and the Gronwall lemma we obtain
$|x_{\eps,\tau}(t)| \le M$ for $t \ge \tau$. Hence,
$$ J(u_{\eps,\tau};\tau,z) \le \int_\tau^\infty e^{-\delta t}
   K_2 (1+M) dt  \ < \ \infty $$
and (\ref{eq:lim-J}) follows.

        Next we prove that the value function satisfies
(\ref{eq:dini-deriv-grad-hjb}) and (\ref{eq:dini+deriv-grad-hjb}). As
already observed at the final part of Section~\ref{sec:monotonicity}, the
value function $V$ satisfies (\ref{eq:HJB-P}), which corresponds to
the first inequality of Proposition~\ref{prop:equiv} (b). Thus, $V$
also satisfies the second inequality, which corresponds to
(\ref{eq:dini-deriv-grad-hjb}) (i.e. $V$ is a Dini sub-solution).

Next we prove that $V$ is a Dini super-solution. It is enough to prove that
$V$ satisfies the first inequality of Proposition~\ref{prop:equiv} (a). Let
$(s,z) \in [0,\infty)\times \R^n$ be given. Under assumptions {\it A1)},
\dots, {\it A6)}, problem $P_{\infty}(s,z)$ has an optimal solution
(see Appendix). Take $(x,u)$ an optimal process for this problem. Thus,
the function $[s,\infty) \ni t \mapsto V(t,x(t)) - \int_{t}^{\infty}
e^{-\delta \sigma} l \big(\sigma,x(\sigma),u(\sigma) \big) d\sigma \in \R$
is constant. The desired inequality follows now from Proposition~%
\ref{prop:monotonicity} (c). \eop
\end{proof}
%
%
%
%------------------------------------------------------------------------------
\section{Verification functions and Minimax results} \label{sec:verif-mm}
\setcounter{equation}{0}

In the first part of this section we use the Dini sub-solutions of the
HJB equation in order to verify optimality of admissible processes.
In the second part we concentrate our attention on some
extremal properties of the value function, which involve the classes of
Dini sub-solutions (respectively super-solutions) of the HJB equation.

\begin{prop} \label{prop:verif}
Let $(\bar{x},\bar{u})\in {\cal D}_{s,z;\infty}$ be given. If assumptions
{\it A1)}, \dots, {\it A5)} hold and there exists a Dini sub-solution
$\Phi$ of the HJB equation such that
\begin{equation} \label{eq:verif}
 \Phi(s,z) = \int_s^{\infty} e^{-\delta t} l(t,\bar x(t),\bar u(t)) dt ,
\end{equation}
then $(\bar x,\bar u)$ is an optimal process for problem $P_{\infty}(s,z)$.
\end{prop}

\begin{proof}
Let $\Phi$ be a function satisfying the assumptions. Given an
arbitrary admissible process $({x},{u}) \in {\cal D}_{s,z;\infty}$, it
follows from Proposition~\ref{prop:equiv} (b) and Proposition~%
\ref{prop:monotonicity} (a) that
$$ t \mapsto \Phi(t,x(t)) - \int_t^{\infty} e^{-\delta r}
   l(r, x(r), u(r)) dr $$
is a monotone non-decreasing function. Thus, for all $t > s$, we have
$$ \Phi(t,x(t)) - \int_t^{\infty} e^{-\delta r} l(r, x(r), u(r)) dr \ge
   \Phi(s,x(s)) - \int_s^{\infty} e^{-\delta r} l(r, x(r), u(r)) dr . $$
From a simple algebraic calculation we deduce
\begin{eqnarray*}
\Phi(t,x(t)) + \int_s^t e^{-\delta r} l(r, x(r), u(r)) dr \ge \Phi(s,x(s)).
\end{eqnarray*}
Taking the limit $t \to \infty$ and using the decay property of $\Phi$,
we conclude that
\begin{eqnarray*}
\int_s^\infty e^{-\delta r} l(r, x(r), u(r)) dr \ge
\Phi(s,x(s)) = \int_s^\infty e^{-\delta r} l(r,\bar x(r),\bar u(r))ds ,
\end{eqnarray*}
proving the optimality of $(\bar x,\bar u)$. \eop
\end{proof} \bigskip

A function $\Phi$ satisfying the assumptions of the above proposition
is called a {\em verification function} for process $(\bar x,\bar u)$.
One should note that the value function $V$ (which is a Dini sub-solution
of the HJB equation from Theorem~\ref{th:vf-dini-sol}) is a verification
function for every optimal process. Summarizing all these facts we have

\begin{prop}
Under assumptions {\it A1)}, \dots, {\it A5)} and {\it \~A6)}, an
admissible process $(\bar{x},\bar{u})\in {\cal D}_{s,z;\infty}$ is
optimal iff there exists
a corresponding verification function.
\end{prop}

        In the sequel we investigate some minimax properties of the value
function with respect to the sets of Dini semi-solutions of the HJB equation.

\begin{prop} \label{prop:minimax}
Under assumptions {\it A1)}, \dots, {\it A6)}, the value function
possesses the following maximal property
$$ V(s,z) \ge \max\{ \Phi(s,z);\ \Phi \mbox{ is a Dini sub-solution} \}. $$
Furthermore, the value function possesses the minimal property
$$ V(s,z) \le \min\{ \Phi(s,z);\ \Phi \mbox{ is a Dini super-solution} \}. $$
\end{prop}

\begin{proof}
It follows from Theorem~\ref{satz:exist} that the value function is
well defined.
Let $\Phi$ be a Dini sub-solution of the HJB equation. From item (b) of
Proposition~\ref{prop:equiv}, we conclude that $\Phi$ satisfies
(\ref{eq:HJB-P}). Thus, from item (a) of Proposition~\ref{prop:monotonicity},
it follows that the function
$$ [s,\infty) \ni t \mapsto \Phi(t,x(t)) - \int_{t}^\infty
   e^{-\delta \sigma} l \big(\sigma,x(\sigma),u(\sigma) \big)
   d\sigma \in \R $$
is non-decreasing along any process $(x,u) \in {\cal D}_{s,z;\infty}$, for
all $(s,z) \in [0,\infty) \times \R^n$. Then, for fixed $(s,z)$ and a
sub-optimal process $(x_\eps,u_\eps) \in {\cal D}_{s,z;\infty}$ we have
\begin{eqnarray*}
 \Phi(t,x_\eps(t)) - \int_{t}^\infty e^{-\delta \sigma}
   l \big(\sigma,x_\eps(\sigma),u_\eps(\sigma) \big) d\sigma & \ge &
   \Phi(s,z) - J(u_\eps; s,z) \\
   & \ge & \Phi(s,z) - V(s,z) - \eps .
\end{eqnarray*}
If we take the limit $t \to \infty$ and use the decay property of $\Phi$,
we obtain $\Phi(s,z) \le V(s,z) + \eps$, proving the desired inequality.

        Now, let $\Phi$ be a Dini super-solution of the HJB equation.
From item (a) of Proposition~\ref{prop:equiv}, we conclude that $\Phi$
satisfies (\ref{eq:HJB-PMM}). Thus, from item (c) of Proposition~%
\ref{prop:monotonicity}, follows the existence of a process $(x,u) \in
{\cal D}_{s,z;\infty}$ (for each $(s,z)$), such that the function
$$ [s,\infty) \ni t \mapsto \Phi(t,x(t)) - \int_{t}^\infty
   e^{-\delta \sigma} l \big(\sigma,x(\sigma),u(\sigma) \big)
   d\sigma \in \R $$
is non-increasing. Then, for fixed $(s,z)$, we have
\begin{eqnarray*}
 \Phi(t,x(t)) - \int_{t}^\infty e^{-\delta \sigma}
   l \big(\sigma,x(\sigma),u(\sigma) \big) d\sigma & \le &
   \Phi(s,z) - J(u_\eps; s,z) \\
   & \le & \Phi(s,z) - V(s,z) .
\end{eqnarray*}
Next, we take the limit $t\to \infty$ in the above inequality and
obtain $\Phi(s,z) \ge V(s,z)$, concluding the proof. \eop
\end{proof} \bigskip

        Since the value function is a Dini solution of the HJB equation,
Proposition~\ref{prop:minimax} can be restated in the following form:

\begin{corol} \label{cor:minimax}
Under assumptions {\it A1)}, \dots, {\it \~A6)}, the value function is
the maximal Dini sub-solution (respectively minimal Dini super-solution)
of the HJB equation.
\end{corol}

        An immediate consequence of Corollary~\ref{cor:minimax} is the
uniqueness of Dini solutions for the HJB equation. %
%
%
%------------------------------------------------------------------------------
\section{Viscosity solutions} \label{sec:visc-sol}\setcounter{equation}{0}

In this section we introduce the concept of viscosity solutions of the
HJB equation. The main result of this section is an equivalence proof
between the concepts of Dini and viscosity solutions of the HJB equation.

\begin{defin} \label{def:visko}
A function $v: [0,\infty) \times \R^n \to \R$ is called {\em viscosity
solution} of the partial differential equation
\begin{equation} \label{eq:hjbg1}
\partial_t v(t,x) \, + \, {\cal H}(t,x,\partial_x v(t,x)) \ = \ 0\, ,
\ (t,x) \in (0,\infty) \times \R^n\,,
\end{equation}
if it satisfies the following conditions:
\begin{itemize}
\item [\it a) ] $v$ is continuous and $v(\tau,z) \to 0$, as
$\tau \to \infty$, uniformly in each compact $K \subset \R^n$;

\item [\it b) ] For every test function $\phi \in C^1((0,\infty) \times
\R^n; \R)$, such that $v - \phi$ has a local maximum at $(t,x) \in (0,\infty)
\times \R^n$, we have
\begin{equation} \label{dev:visc-sub}
\partial_t \phi(t,x)\, +\, {\cal H}(t,x,\partial_x\phi(t,x))\ \ge\ 0 ;
\end{equation}
\item [\it c) ] For every test function $\phi \in C^1((0,\infty) \times
\R^n; \R)$, such that $v - \phi$ has a local minimum at $(t,x) \in (0,\infty)
\times \R^n$, we have
\begin{equation} \label{dev:visc-sup}
\partial_t \phi(t,x)\, +\, {\cal H}(t,x,\partial_x\phi(t,x))\ \le\ 0 .
\end{equation}
\end{itemize}
If the function $v$ satisfies only {\it a)} and {\it b)}, it is called {\em
viscosity sub-solution}. If $V$ satisfies only {\it a)} and {\it c)}, it is
called {\em viscosity super-solution}.
\end{defin}

        Next we verify that a function is a viscosity solution iff it
is a Dini solution of the HJB equation.

\begin{theorem} \label{theor:dini-visc}
The following assertions hold true: \\
(i) A function $v$ is a viscosity sub-solution of the HJB equation iff
\begin{enumerate}
\item[(a)] for all $(s,z) \in [0,\infty) \times \R^n$,
\begin{equation}\label{super-HJB-D}
\big\{ \alpha + {\cal H}(s,z,\xi) \big\} \ \ge \ 0, \quad \forall
(\alpha,\xi)\in \partial^D v(s,z);
\end{equation}
\item[(b)] $\lim\limits_{s\uparrow \infty} v(s,z) = 0$ uniformly
for $z$ in compact subsets of $\R^n$.
\end{enumerate}
(ii) A function $v$ is a viscosity super-solution of the HJB equation iff
it satisfies assertion (b) of item (i) and, for all $(s,z) \in [0,\infty)
\times \R^n$,
we have
\begin{equation}
\big\{ \alpha + {\cal H}(s,z,\xi) \big\} \ \le \ 0, \quad \forall
(\alpha,\xi)\in \partial_D v(s,z).
\end{equation}
\end{theorem}

\begin{proof}
We begin by proving (i). In the proof of the necessity, as well as in
the proof of sufficiency, the decay condition is fulfilled by hypothesis.
Therefore, it remains only to prove the equivalence between inequalities
(\ref{super-HJB-D}) and (\ref{dev:visc-sub}).

Let $v$ be a continuous function satisfying (\ref{super-HJB-D}) and $\phi$
be an arbitrary $C^1$ test function such that $v - \phi$ has a local
maximum at $(s,z) \in [0,\infty) \times \R^n$. Consequently, $-(v-\phi)$
has a local minimum at $(s,z)$, and we have $0 \in \partial_D (-v + \phi)
(s,z)$. Thus, $-\nabla \phi(s,z) \in \partial_D(-v) (s,z)$. However, this
is equivalent to $\nabla \phi(s,z) \in \partial^D v(s,z)$. Now, taking
$(\alpha,\xi) = \nabla \phi(s,z)$ in (\ref{super-HJB-D}), we conclude that
$\phi$ satisfies (\ref{dev:visc-sub}).

Next we prove the converse. Let $v$ be a viscosity sub-solution of
the HJB equation. Given $(\alpha,\xi) \in \partial^D v$, we have
$(-\alpha,-\xi) \in \partial_D (-v)$. From \cite[Proposition 3.4.12]{CLSW}
follows the existence of a $C^1$ function $-\phi$ such that $\nabla
(-\phi)(s,z) = (-\alpha,-\beta)$ and $-v + \phi$ has a local minimum at
$(s,z)$ (i.e., $v - \phi$ has a local maximum at $(s,z)$). Now, since
$v$ is a viscosity sub-solution, (\ref{super-HJB-D}) follows from
(\ref{dev:visc-sub}) for this particular $\phi$.

The proof of assertion (ii) is analogous and will be omitted. \eop
\end{proof} \bigskip

        As an immediate consequence of Theorem~\ref{theor:dini-visc}, we
conclude that the value function is the unique viscosity solution of
the HJB equation with a particular decay property. This theorem also
characterizes Dini (semi) solutions, which were in focus of our attention
in this paper, as an important tool to investigate nonautonomous optimal
control problems of infinite horizon type. This last assertion is based
on the knowledge that both concepts of solutions, Dini and viscosity,
coincide, agreeing with the theory in finite or (autonomous) infinite
horizon problems.
%
%
%
%------------------------------------------------------------------------------
%\newpage

%
%------------------------------------------------------------------------------

\section{Appendix}\setcounter{equation}{0}

In this appendix we discuss the issue of existence of an optimal
control process for $P_{\infty}(\tau,z)$. The method of proof is
standard, therefore we omit the proof and we make some comments
instead.

\begin{theorem} \label{satz:exist}
Suppose {\it A1), \dots, A6)} are satisfied. Then there exists an
optimal process.
\end{theorem}

The existence result in Theorem~\ref{satz:exist} is under compactness
of the set of admissible control actions and convexity of the so called
extended velocity vector.
We make the observation that the arguments of the proof are the
usual in which we start with a minimizing sequence and under
``compactness" of the trajectories on compact subintervals of
$[\tau,\infty)$ we extract a subsequence, which converges to an
absolutely continuous function uniformly on compact subintervals.
This function, with the help of Filippov's lemma, is proved to be
a state trajectory corresponding to some admissible control.
Together, they form a process that is optimal for $P_{\infty}(\tau,z)$.
These are standard arguments in finite time interval that can be
modified accordingly to be reproduced for infinite time horizon.

\end{document}